\documentclass{amsart}

\usepackage[utf8]{inputenc}
\usepackage[T1]{fontenc}
\usepackage{lmodern}

\usepackage[svgnames]{xcolor}

\usepackage[english]{babel}

\usepackage{amssymb} 
\usepackage[colorlinks,linkcolor=Maroon,citecolor=Teal,urlcolor=Navy]{hyperref} 
\usepackage{amsrefs}
\usepackage[capitalise]{cleveref}

\newtheorem{theorem}{Theorem}
\newtheorem{lemma}[theorem]{Lemma}
\newtheorem{proposition}[theorem]{Proposition}

\theoremstyle{definition}

\theoremstyle{remark}

\numberwithin{equation}{section}

\newcommand{\R}{\mathbb{R}}

\newcommand{\cc}{\subset\subset}
\newcommand{\abs}[1]{\lvert#1\rvert}
\newcommand{\norm}[1]{\lVert#1\rVert}
\newcommand{\dx}[1]{\mathrm{d}#1}
\newcommand{\grad}{\nabla}
\DeclareMathOperator{\tr}{tr}

\newcommand{\D}{\mathcal{D}}
\newcommand{\loc}{\mathrm{loc}}
\newcommand{\gl}{\mathrm{gl}}
\newcommand{\so}{\mathrm{so}}
\newcommand{\Om}{\mathrm{O}}
\newcommand{\SO}{\mathrm{SO}}
\newcommand{\Sym}{\mathrm{Sym}}
\newcommand{\Symp}{\Sym^+}

\title[Optimal regularity for Pfaffian systems and the FTST]{Optimal regularity for two-dimensional \\Pfaffian systems and the fundamental \\theorem of surface theory}
\author{Florian Litzinger}
\address{School of Mathematical Sciences, Queen Mary University of London, London E1\,4NS, United Kingdom}
\email{f.g.litzinger@qmul.ac.uk}
\date{January 20, 2020}

\makeatletter
\hypersetup{pdftitle=\@title,pdfauthor=\authors}
\makeatother

\begin{document}
\begin{abstract}
  We prove that a Pfaffian system with coefficients in the critical space $L^2_\mathrm{loc}$ on a simply connected open subset of $\mathbb{R}^2$ has a non-trivial solution in $W^{1,2}_\mathrm{loc}$ if the coefficients are antisymmetric and satisfy a compatibility condition. As an application of this result, we show that the fundamental theorem of surface theory holds for prescribed first and second fundamental forms of optimal regularity in the classes $W^{1,2}_\mathrm{loc}$ and $L^2_\mathrm{loc}$, respectively, that satisfy a compatibility condition equivalent to the Gauss--Codazzi--Mainardi equations. Finally, we give a weak compactness theorem for surface immersions in the class $W^{2,2}_\mathrm{loc}$.
  \end{abstract}

\maketitle

\section{Introduction}
\label{sec:Introduction}

In an open set $U \subset \R^2$, we consider a Pfaffian system of the form
\begin{equation}\label{eq:PfaffSystem}
  \grad P = P \Omega,
\end{equation}
where $P$ is a matrix-valued function and $\Omega$ is a given matrix-valued one-form. Local existence of a non-trivial solution $P$ to this partial differential equation, and its regularity, manifestly depend on the regularity properties of the coefficients. It is a classical result that a twice continuously differentiable solution exists if every component $\Omega_i$ is continuously differentiable and they satisfy the compatibility condition
\begin{equation}\label{eq:Compatibility}
  \partial_i \Omega_j - \partial_j \Omega_i = \Omega_j \Omega_i - \Omega_i \Omega_j.
\end{equation}
One objective of the present paper is to show the corresponding result for solutions $P \in W^{1,2}_\loc$ and coefficients $\Omega \in L^2$ satisfying an additional structural assumption. This is the case of least possible regularity for an equation such as \cref{eq:Compatibility} to make sense in an integrated form. We then have
\begin{theorem}
  \label{thm:PfaffGlobal}
  Let $U \subset \R^2$ be a connected and simply connected open set and let $\Omega \in L^2(U, \so(m) \otimes \wedge^1 \R^2)$ satisfy the compatibility condition, \cref{eq:Compatibility}, in the distributional sense. Then there exists $P \in W^{1,2}_\loc(U, \SO(m))$ such that $\grad P = P \Omega$ in $U$.
  Moreover, any two such solutions $P_0$, $P_1$ are related by $P_0 = C P_1$ with a constant $C \in \SO(m)$.
\end{theorem}

Over the years, there have been several incremental improvements to the classical theory. In particular, Hartman and Wintner \cite{hartman1950fundamental} showed that the above existence result holds if the given form $\Omega$ is continuous, with a continuously differentiable solution $P$. Following this, Mardare \cites{mardare2003fundamental,mardare2005pfaff} first showed the existence of a solution $P$ to \cref{eq:PfaffSystem} in the Sobolev class $W^{1,\infty}_\loc$ for locally essentially bounded coefficients and later improved the theorem to hold in the class $W^{1,p}_\loc$ for $\Omega \in L^p_\loc$, where $p>2$. It is important to note that without any further structural assumptions on the coefficients $\Omega$, this result has been demonstrated to be optimal \cite{mardare2005pfaff}. However, once one supposes that the components of the matrix-valued one-form $\Omega$ be \emph{antisymmetric}, it is possible to improve the regularity to the critical case that is \cref{thm:PfaffGlobal}.

Meanwhile, there have been developments in the theory of non-linear PDE that attempt to exploit a particular structure of the equation in order to gain additional regularity of the solution beyond what would usually be expected; and these compensated compactness methods \cites{coifman1993compensated,riviere2007conservation,wente1969existence} have been markedly successful in that regard. In particular, in his 2007 paper, Rivière \cite{riviere2007conservation} provided a proof of the regularity of two-dimensional weakly harmonic maps, from which we recall an important intermediate result:
\begin{lemma}[Rivière \cite{riviere2007conservation}, Lemma A.3; Schikorra \cite{schikorra2010remark}]   \label{lem:Riviere}
  Let $U \subset \R^2$ be a contractible bounded regular domain and let $\Omega \in L^2(U, \so(m) \otimes \wedge^1 \R^2)$. Then there exist $\xi \in W_0^{1,2}(U, \so(m))$ and $P \in W^{1,2}(U, \SO(m))$ such that
  \begin{align}
    P^{-1} \grad P + P^{-1} \Omega P &= \grad^\perp \xi, \\
    \norm{\grad \xi}^2_{L^2} + \norm{\grad P}^2_{L^2} &\leq C(m) \norm{\Omega}^2_{L^2}.
  \end{align}
\end{lemma}
Thanks to the Riemann mapping theorem, \cref{lem:Riviere} also holds true if $U \subset \R^2$ is an open, connected, and simply connected bounded set with sufficiently smooth boundary.
While the techniques employed in the original proof \cite{riviere2007conservation} are quite involved, Schikorra \cite{schikorra2010remark} gave an alternative proof using variational methods, which in addition removes the need for a smallness condition on $\Omega$.

The above result is of particular interest to us because the given form $\Omega$ is only assumed to be square-integrable. In order to achieve existence and regularity of the solution $P \in W^{1,2}$, the additional structure assumed, that is, the antisymmetry of each $\Omega_i$, is utilized in a crucial way. In the same vein, it is this additional structural assumption that enables us to employ Rivière's lemma to extend the previous results on the solvability of the above Pfaffian system in \cref{eq:PfaffSystem} to the critical $p=2$ case.

The possibility of finding a solution to this Pfaffian system, in turn, has been an essential ingredient in the proof of weak versions of the fundamental theorem of surface theory. As for Pfaffian systems, there have been incremental improvements to this classical geometric result. The theorem answers the question of whether it is possible to find an immersion of a surface in three-dimensional space with prescribed first and second fundamental forms---this turns out to be true if, and only if, the fundamental forms satisfy the Gauss--Codazzi--Mainardi equations. We obtain the following
\begin{theorem}\label{thm:FTST}
  Let $U$ be a connected and simply connected open subset of $\R^2$ and let $(a_{ij}) \in W^{1,2}_\loc(U, \Symp(2)) \cap L^\infty_\loc(U, \Symp(2))$ and $(b_{ij}) \in L^2_\loc(U, \Sym(2))$ be given.
  Suppose that the eigenvalues of $(a_{ij})$ are locally uniformly bounded from below and that the matrix fields $(a_{ij})$, $(b_{ij})$ are such that
  \begin{equation}
    \partial_1 \Omega_2 - \partial_2 \Omega_1 = \Omega_2 \Omega_1 - \Omega_1 \Omega_2,
  \end{equation}
  where $\Omega \in L^2_\loc(U, \so(3) \otimes \wedge^1 \R^2)$ is given by the following sequence of definitions, see also \cref{sub:RegularityCoefficients}:
  \begin{align*}
    (a^{i j}) &= \frac{1}{a_{11}a_{22} - a_{12}a_{21}} \begin{pmatrix}
      a_{22} & -a_{12} \\ -a_{21} & a_{11}
    \end{pmatrix}, \\
    b_ i^j &= a^{j k} b_{i k}, \\
    \Gamma^k_{i j} &= \frac{1}{2} a^{k \ell} (\partial_j a_{i \ell} + \partial_ i a_{j \ell} - \partial_ \ell a_{i j}), \\
    G &= \begin{pmatrix}
      a_{11} & a_{12} & 0 \\ a_{21} & a_{22} & 0 \\ 0 & 0 & 1
    \end{pmatrix}^\frac{1}{2}, \\
    \Gamma_ i &= \begin{pmatrix}
      \Gamma_{i 1}^1 & \Gamma_{i 2}^1 & -b_ i^1 \\
      \Gamma_{i 1}^2 & \Gamma_{i 2}^2 & -b_ i^2 \\
      b_{i 1} & b_{i 2} & 0
    \end{pmatrix}, \\
    \Omega_ i &= (G \Gamma_ i - \partial_ i G) G^{-1}.
  \end{align*}

  Then there exists an immersion $\theta \in W^{2,2}_\loc(U, \R^3)$ such that
  \begin{alignat}{2}
    a_{ij} &= \partial_i \theta \cdot \partial_j \theta \qquad &&\text{in $W^{1,2}_\loc(U)$}, \\
    b_{ij} &= \partial_{ij} \theta \cdot \frac{\partial_1 \theta \times \partial_2 \theta}{\abs{\partial_1 \theta \times \partial_2 \theta}} \qquad &&\text{in $L^2_\loc(U)$}.
  \end{alignat}
  Moreover, the map $\theta$ is unique in $W^{2,2}_\loc(U, \R^3)$ up to proper isometries of $\R^3$.
\end{theorem}
We remark that the compatibility condition assumed in the theorem is in fact equivalent to the Gauss--Codazzi--Mainardi equations, see \cref{prop:Equivalence}. As for the Pfaffian system mentioned above, one needs to consider the compatibility equations in the distributional sense.

In the works mentioned above \cites{hartman1950fundamental,mardare2003fundamental,mardare2005pfaff}, the fundamental theorem of surface theory has been extended to hold true for, finally, first and second fundamental forms in the classes $W^{1,p}_\loc$ and $L^p_\loc$, respectively, where $p>2$.
The method of proof, whose lines we also follow in this paper, is the following: First, a Pfaffian system as in \cref{eq:PfaffSystem} is solved for a proper orthogonal matrix field $P$, and then the sought-after surface immersion is found by means of a weak version of the Poincaré lemma, solving the equation $\grad \theta = PG$, where $G$ is the matrix square root of the three-dimensional extension of the given metric. Since the Poincaré lemma is known to hold for all $p \geq 1$ (see \cref{lem:Poincare}), the premier challenge in extending the fundamental theorem of surface theory to the critical exponent $p=2$ lies in the extension of the corresponding existence theorem on Pfaffian systems.

Therefore, in order to be able to apply \cref{thm:PfaffGlobal}, an appropriate \emph{antisymmetric} matrix-valued one-form $\Omega$ of coefficients of the Pfaffian system has to be constructed as above. While the connection form $\Gamma$ does not possess this property in an arbitrary frame, it is known to be antisymmetric in an orthonormal frame. This approach to the fundamental theorem of surface theory, via an antisymmetric field of coefficients, has previously been introduced by Ciarlet, Gratie, and C.~Mardare \cite{ciarlet2008new}, who identified the solution $P$ of the Pfaffian system as the rotation field appearing in the polar factorization of the gradient of the three-dimensional extension of the immersion $\theta$.

As a consequence of our approach, we finally obtain a weak rigidity of the compatibility equation and a weak compactness theorem for surface immersions in the class $W^{2,2}_\loc$. \begin{theorem}\label{thm:Compactness}
  Let $\{\theta^k\} \subset W^{2,2}_\loc(U, \R^3)$ be a uniformly bounded sequence of immersions with corresponding sequences of first and second fundamental forms denoted by $\{(a_{ij})^k\}$ and $\{(b_{ij})^k\}$, respectively.
  Suppose that $\partial_i \theta^k \in W^{1,2}_\loc \cap L^\infty_\loc$ and that the first fundamental forms $(a_{ij})^k$, $a_{ij}^k = \partial_i \theta^k \cdot \partial_j \theta^k$, have eigenvalues bounded from below uniformly in the domain $U$ and in $k$.
  Then, after passing to subsequences, $\{\theta^k\}$ converges weakly in $W^{2,2}_\loc$ to an immersion $\theta \in W^{2,2}_\loc(U, \R^3)$, whose first and second fundamental forms $(a_{ij})$, $(b_{ij})$ are limit points of the sequences $\{(a_{ij})^k\}$, $\{(b_{ij})^k\}$ in the $W^{1,2}_\loc$- and $L^2_\loc$-topologies, respectively.
\end{theorem}
In the context of immersions of Riemannian manifolds, results in this spirit already appeared in a recent work by Chen and Li \cite{chen2018global}. Moreover, sequences of weak immersions have previously been investigated without any assumptions about the first fundamental form, supposing instead a uniform bound on the $L^2$-norm of the second fundamental form---see the paper of Laurain and Rivière \cite{laurain2018optimal} and the references therein.

This paper is structured as follows: In \cref{sec:Notation} we introduce the notation that is used throughout this article.
After that, we prove in \cref{sec:Pfaffian systems} that the optimal existence theorem for Pfaffian systems is a consequence of the aforementioned lemma of Rivière. We also extend the theorem from the unit disk to arbitrary simply connected open subsets of $\R^2$.
Thereafter, in \cref{sec:Application to FTST}, we apply this result to the optimal regularity case of the fundamental theorem of surface theory, mostly following along the lines of previous approaches \cites{ciarlet2008new,mardare2005pfaff}.
In \cref{sec:Compactness theorem}, finally, we conclude the paper by demonstrating the weak compactness of $W^{2,2}_\loc$-immersions.

\subsection*{Acknowledgements}
The author would like to express his gratitude to his supervisor Huy The Nguyen for suggesting the problem and his constant support, Tristan Rivière for pointing out references to previous results on weak compactness of $W^{2,2}$-immersions, Ben Sharp for providing a crucial argument in the proof of \cref{prop:Pfaff}, and Gianmichele Di Matteo for many interesting discussions.

\section{Notations and Preliminaries}
\label{sec:Notation}

Throughout this paper, let $U$ be an open, connected and simply connected subset of $\R^2$. A continuously differentiable mapping $\theta: U \to \R^3$ is called an immersion if the vectors $\partial_i \theta(y)$, $i=1,2$, are linearly independent for all $y \in U$.

We denote the set of real matrices of size $n \times n$ by $\gl(n)$, the set of symmetric matrices by $\Sym(n)$, the set of symmetric positive definite matrices by $\Symp(n)$, the set of antisymmetric matrices by $\so(n)$, the set of orthogonal matrices by $\Om(n)$, and the set of proper orthogonal matrices by $\SO(n)$.
We write the space of $\so(n)$-valued one-forms on $\R^2$ as $\so(n) \otimes \wedge^1 \R^2$. The components of $\Omega \in \so(n) \otimes \wedge^1 \R^2$ are denoted by $\Omega_i$, $i=1,2$, such that $\Omega_i \in \so(n)$.

Moreover, we denote the elements of a matrix $A \in \gl(n)$ by $a_{ij}$, $i,j=1,\dots,n$, such that $A = (a_{ij})$, and the $j$-th column of $A$ is denoted by $A_{(j)} = a_j$. The inverse $A^{-1}$ of $A$ is denoted by $(a^{ij})$ and the transpose of $A$ by $A^T = (a_{ji})$.
We enumerate the real eigenvalues of $A \in \Sym(n)$ as $\lambda_1(A) \leq \dots \leq \lambda_n(A)$ and with any $A \in \Symp(n)$ we associate the unique matrix square root $A^{\frac{1}{2}}$.

We write $\D(U)$ for the space of smooth functions with compact support contained in $U$ and $\D'(U)$ for the space of distributions over $U$.
As usual, we denote the Lebesgue spaces by $L^p(U)$, $1 \leq p \leq \infty$, and the Sobolev spaces of (equivalence classes of) weakly differentiable functions by $W^{k,p}(U)$, $k=0,1,\dots$, $1 \leq p \leq \infty$. The closure of $\D(U)$ in $W^{1,2}(U)$ is denoted by $W^{1,2}_0(U)$. Furthermore, we write
\begin{equation*}
  W^{k,p}_\loc(U) = \{ T \in \D'(U) : T \in W^{k,p}(V) \text{ for all open sets } V \cc U \}.
\end{equation*}
Whenever $X$ is a finite-dimensional space, let $\D(U, X)$, $L^p(U, X)$, and $W^{k,p}(U, X)$ designate the spaces of $X$-valued objects whose components belong to $\D(U)$, $L^p(U)$, and $W^{k,p}(U)$, respectively. We shall omit the additional symbol if it is implied by the context.

We note that the space $W^{1,2}(B) \cap L^\infty(B)$ is an algebra for all open balls $B \cc U$, so that $fg \in W^{1,2}_\loc(U) \cap L^\infty_\loc(U)$ whenever $f,g \in W^{1,2}_\loc(U) \cap L^\infty_\loc(U)$.

For later use, we recall the following weak version of the Poincaré lemma.

\begin{lemma}[Mardare \cite{mardare2007systems}, Theorem 6.5]\label{lem:Poincare}
  Let $U$ be a connected and simply connected open subset of $\R^2$ and let $p \geq 1$. Let $f_i \in L^p_\loc(U)$, $i=1,2$, be functions that satisfy
  \begin{equation}
    \partial_1 f_2 = \partial_2 f_1 \qquad \text{in $\D'(U)$}.
  \end{equation}
  Then there exists a function $\theta \in W_\loc^{1,p}(U)$, unique up to an additive constant, such that
  \begin{equation}
    \partial_i \theta = f_i \qquad \text{in $L^p_\loc(U)$}.
  \end{equation}
\end{lemma}

Finally, we remark that the Pfaffian system in \cref{eq:PfaffSystem} studied in this paper can be understood in the following way: We interpret $\Omega \in \so(m) \otimes \wedge^1 \R^2$ as a tensor $\Omega^i_{j \ell}$ that is antisymmetric in $i$ and $j$. \Cref{eq:PfaffSystem} then reads, for $\ell=1,2$,
\begin{equation*}
  \partial_\ell P = P \Omega_\ell,
\end{equation*}
that is, assuming the summation convention,
\begin{equation*}
  \partial_\ell P^i_j = P^i_k \Omega^k_{j \ell}.
\end{equation*}

\section{Pfaffian Systems with Coefficients in $L^2$}
\label{sec:Pfaffian systems}

This section is devoted to the proof of \cref{thm:PfaffGlobal}. Building upon \cref{lem:Riviere}, we first show the following
\begin{proposition}
  \label{prop:Pfaff}
  Let $U$ and $\Omega$ be as in \cref{lem:Riviere} and let $\Omega$ satisfy the compatibility equation
  \begin{equation}\label{eq:Pfaff_Compatibility}
    \partial_i \Omega_j - \partial_j \Omega_i = \Omega_j \Omega_i - \Omega_i \Omega_j.
  \end{equation}
  Then there exists $P \in W^{1,2}(U, \SO(m))$ such that
  \begin{equation}\label{eq:Pfaff_eq}
    \grad P + \Omega P = 0.
  \end{equation}
  Moreover, if $P_0$ and $P_1$ are two such solutions then there exists a constant $C \in \SO(m)$ such that
  \begin{equation}
    P_0 = P_1 C.
  \end{equation}
\end{proposition}
\begin{proof}
  By \cref{lem:Riviere}, there exist $\xi \in W_0^{1,2}(U, \so(m))$ and $P \in W^{1,2}(U, \SO(m))$ such that
  \begin{equation}
    P^{-1} \grad P + P^{-1} \Omega P = \grad^\perp \xi.
  \end{equation}
  We obtain, using the compatibility equation \labelcref{eq:Pfaff_Compatibility},
  \begin{align*}
    \nabla^\perp (P \nabla^\perp \xi) &= \nabla^\perp \nabla P + \nabla^\perp (\Omega P) \\
    \nonumber &= \nabla^\perp \Omega P + \Omega \nabla^\perp P \\
    \nonumber &= -\Omega_1 P \partial_1 \xi - \Omega_2 P \partial_2 \xi,
  \end{align*}
  whence
  \begin{equation*}
    \nabla^\perp P \cdot \nabla^\perp \xi + P \Delta \xi = -\Omega P \nabla \xi,
  \end{equation*}
  and thus
  \begin{align*}
    P \Delta \xi &= - \nabla^\perp P \cdot \nabla^\perp \xi - \Omega P \nabla \xi \\
    \nonumber &= -P \nabla^\perp \xi \cdot \nabla \xi,
  \end{align*}
  or equivalently
  \begin{equation}\label{eq:PoissonXi}
    \Delta \xi = (\partial_2 \xi) (\partial_1 \xi) - (\partial_1 \xi) (\partial_2 \xi).
  \end{equation}

  While the right hand side of this equation is not necessarily equal to zero, we claim that \cref{eq:PoissonXi} does imply that $\xi \equiv 0$, using that $\xi|_{\partial U} = 0$.
  Indeed, this follows directly from a theorem of Wente \cite{wente1975differential}, thanks to the fact that, by a result of Müller and Schikorra \cite{muller2009boundary}, $\xi \in W_0^{1,2}(U, \so(m))$ is continuous in $\bar{U}$. This yields \cref{eq:Pfaff_eq}.

  Now suppose that $P_0, P_1 \in W^{1,2}(U, \SO(m))$ solve
  \begin{align*}
    \nabla P_0 + \Omega P_0 &= 0, \\
    \nabla P_1 + \Omega P_1 &= 0
  \end{align*}
  in $U$, respectively. Since $\nabla(P_1^{-1}) = - P_1^{-1} (\nabla P_1) P_1^{-1}$, we have
  \begin{align*}
    \nabla(P_1^{-1} P_0) &= \nabla(P_1^{1}) P_0 + P_1^{-1} \nabla P_0 \\
    \nonumber &= - P_1^{-1} (\nabla P_1) P_1^{-1} P_0 + P_1^{-1} \nabla P_0 \\
    \nonumber &= P_1^{-1} \Omega P_1 P_1^{-1} P_0 - P_1^{-1} \Omega P_0 \\
    \nonumber &= 0.
  \end{align*}
  Thus $P_1^{-1} P_0 = C$, that is, $P_0 = P_1 C$, a constant invertible matrix. We also have $C^T C = P_0^{-1} P_1 P_1^{-1} P_0 = I$ and $\det C = (\det P_1)^{-1} \det P_0 = 1$, whereby $C \in \SO(m)$.
\end{proof}

In order to prove \cref{thm:PfaffGlobal}, it remains to extend the statement of \cref{prop:Pfaff} to any connected and simply connected open set $U$.
In fact, this amounts to the construction of a global solution by gluing together local solutions as in, e.\,g., the proof of the Poincaré lemma with little regularity given by Mardare \cite{mardare2008poincare}, the details of which we thus omit. It should be noted that this construction is not limited to the two-dimensional case. Alternatively, one may use the fact that, as mentioned above, simply connected domains in the plane are contractible. As a result, \cref{thm:PfaffGlobal} follows readily after transposition and using that $\Omega_i \in \so(m)$.

\section{Application to the Fundamental Theorem of Surface Theory}
\label{sec:Application to FTST}

In this section, we shall apply \cref{thm:PfaffGlobal} in order to prove the existence of a $W^{2,2}_\loc$-immersion of a surface with prescribed first and second fundamental forms in the classes $W^{1,2}_\loc$ and $L^2_\loc$, respectively. First, we motivate the definition of appropriate antisymmetric matrix fields $\Omega_i$ that serve as the coefficients of a Pfaffian system. After that, we show that the quantities derived from the given matrix fields that are to be realized as fundamental forms of a surface possess the required regularity. We then prove \cref{thm:FTST}. Lastly, we demonstrate that the compatibility equation satisfied by the matrix fields $\Omega_i$ is equivalent to the Gauss--Codazzi--Mainardi equations in the present setting.

\subsection{Derivation of antisymmetric coefficients}

Following the exposition in Clelland \cite{clelland2017frenet}, we derive the antisymmetric quantities that have previously been introduced by Ciarlet, Gratie, and C. Mardare \cite{ciarlet2008new}, but this time from the viewpoint of Cartan geometry.

Let $U \subset \R^2$ be open, connected, and simply connected and let $\theta: U \to (\R^3, \langle \cdot\,, \cdot \rangle)$ be a smooth immersion whose image $\Sigma = \theta(U)$ is a regular surface.
Furthermore, let $\tilde\theta: U \to E(3)$, $\tilde\theta(x) = (\theta(x); e_1(x), e_2(x), e_3(x))$, where $E(3)$ is the Euclidean group, be an adapted frame field. This means that for each $x \in U$, $(e_1(x), e_2(x), e_3(x))$ is an oriented orthonormal basis of $T_{\theta(x)}\R^3$ and $e_3(x)$ is orthogonal to $T_{\theta(x)}\Sigma$.

We define scalar-valued one-forms $(\omega^i, \omega^i_j)$ on $E(3)$ by
\begin{align}
  \dx{\theta} &= e_i \omega^i, \\
  \dx{e_i} &= e_j \omega^j_i.
\end{align}
They have the properties
\begin{align*}
  \omega^i(e_j) &= \delta^i_j, \\
  \omega^j_i &= - \omega^i_j,
\end{align*}
and they satisfy the Cartan structure equations
\begin{align}
  \dx{\omega^i} &= - \omega^i_j \wedge \omega^j, \\
  \dx{\omega^i_j} &= - \omega^i_k \wedge \omega^k_j.
\end{align}

The Maurer--Cartan form on $E(3)$ is given by
\begin{equation*}
  \Omega = \begin{pmatrix}
    0 & 0 & 0 & 0 \\
    \omega^1 & \omega^1_1 & \omega^1_2 & \omega^1_3 \\
    \omega^2 & \omega^2_1 & \omega^2_2 & \omega^2_3 \\
    \omega^3 & \omega^3_1 & \omega^3_2 & \omega^3_3
  \end{pmatrix}
\end{equation*}
and the Cartan structure equations are equivalent to the Maurer--Cartan equation
\begin{equation}
  \label{eqn:MaurerCartan}
  \dx{\Omega} = - \Omega \wedge \Omega.
\end{equation}
The one-forms $(\omega^i, \omega^i_j)$ are collectively referred to as Maurer--Cartan forms as well.

Let $(\bar\omega^i, \bar\omega^i_j) = (\tilde \theta^* \omega^i, \tilde \theta^* \omega^i_j)$ be the pullbacks on $U$. Then $\bar\omega^3 = 0$ and $\bar \omega^1$, $\bar \omega^2$ form a basis for the one-forms on $U$.
The first fundamental form on $T U$ is given by
\begin{equation}
  \mathrm{I} = (\bar \omega^1)^2 + (\bar \omega^2)^2,
\end{equation}
and the second fundamental form by
\begin{equation}
  \mathrm{II} = \bar \omega^3_1 \bar \omega^1 + \bar \omega^3_2 \bar \omega^2 = h_{11} (\bar \omega^1)^2 + 2 h_{12} \bar \omega^1 \bar \omega^2 + h_{22} (\bar \omega^2)^2,
\end{equation}
where $h_{11}, h_{12}, h_{22}$ are such that
\begin{equation*}
  \begin{pmatrix}
    \bar \omega^3_1 \\ \bar \omega^3_2
  \end{pmatrix} = \begin{pmatrix}
    h_{11} & h_{12} \\ h_{12} & h_{22}
  \end{pmatrix}
  \begin{pmatrix}
    \bar \omega^1 \\ \bar \omega^2
  \end{pmatrix}.
\end{equation*}

Hence\begin{alignat*}{2}
  \bar \omega^3_k &= h_{kj} \bar \omega^j, &\qquad 1 &\leq j,k \leq 2, \\
  \bar \omega^j_\ell &= \Gamma_{k\ell}^j \bar \omega^k, &\qquad 1 &\leq j,k,\ell \leq 2, \\
  \bar \omega^i_i &= - \bar \omega^i_i = 0, &\qquad i &= 1,2,3.
\end{alignat*}
Moreover, the Maurer--Cartan equation \labelcref{eqn:MaurerCartan} is equivalent to the Gauss--Codazzi--Mainardi equations, which read in this notation as follows:
\begin{align}
  \dx{\bar \omega^1_2} &= \bar \omega^3_1 \wedge \bar \omega^3_2, \\
  \dx{\bar \omega^3_1} &= \bar \omega^3_2 \wedge \bar \omega^1_2, \\
  \dx{\bar \omega^3_2} &= - \bar \omega^3_1 \wedge \bar \omega^1_2.
\end{align}

If we write
\begin{equation*}
  \underline{\omega} = \begin{pmatrix}
  0 & \bar\omega^1_2 & \bar\omega^1_3 \\
  \bar\omega^2_1 & 0 & \bar\omega^2_3 \\
  \bar\omega^3_1 & \bar\omega^3_2 & 0
  \end{pmatrix},
\end{equation*}
then we may define
\begin{equation*}
  \Gamma_i := \underline{\omega}(e_i) = \begin{pmatrix}
    0 & \Gamma_{i2}^1 & -h_{1i} \\
    \Gamma_{i1}^2 & 0 & -h_{2i} \\
    h_{1i} & h_{2i} & 0
  \end{pmatrix}.
\end{equation*}

Now, given a metric $\bar g$ on $\Sigma$ and an orthonormal frame $e = (e_1, e_2, e_3)$, we set
\begin{equation*}
  g = \begin{pmatrix}
    \bar g_{11} & \bar g_{12} & 0 \\
    \bar g_{21} & \bar g_{22} & 0 \\
    0 & 0 & 1
  \end{pmatrix}^{\frac{1}{2}}.
\end{equation*}
Defining the frame $e' = eg^{-1}$, which is orthonormal with respect to $g^2$, the Maurer--Cartan form in this frame is given by means of the gauge transformation
\begin{equation}
    \underline{\omega}' = (g \underline{\omega} + \dx{g}) g^{-1},
\end{equation}
which implies in components that
\begin{equation}
  \Gamma'_i = (g \Gamma_i - \partial_i g) g^{-1}.
\end{equation}
Differentiating the orthonormality condition for the frame $e'$ with respect to the metric $g^2$, we see that the connection coefficients $\Gamma'_i$ must be antisymmetric.

\subsection{Regularity of coefficients}
\label{sub:RegularityCoefficients}
Let $(a_{i j}) \in W ^{1,2}_\loc (U, \Symp(2)) \cap L^\infty_\loc (U, \Symp(2))$ and $(b_{i j}) \in L^2_\loc (U, \Sym(2))$ and define
\begingroup
\allowdisplaybreaks
  \begin{align}
    (a^{i j}) &= \frac{1}{a_{11}a_{22} - a_{12}a_{21}} \begin{pmatrix}
      a_{22} & -a_{12} \\ -a_{21} & a_{11}
    \end{pmatrix}, \\
    b_ i^j &= a^{j k} b_{i k}, \\
    \Gamma^k_{i j} &= \frac{1}{2} a^{k \ell} (\partial_j a_{i \ell} + \partial_ i a_{j \ell} - \partial_ \ell a_{i j}), \\
    G &= \begin{pmatrix}
      a_{11} & a_{12} & 0 \\ a_{21} & a_{22} & 0 \\ 0 & 0 & 1
    \end{pmatrix}^\frac{1}{2}, \\
    \Gamma_ i &= \begin{pmatrix}
      \Gamma_{i 1}^1 & \Gamma_{i 2}^1 & -b_ i^1 \\
      \Gamma_{i 1}^2 & \Gamma_{i 2}^2 & -b_ i^2 \\
      b_{i 1} & b_{i 2} & 0
    \end{pmatrix}, \\
    \Omega_ i &= (G \Gamma_ i - \partial_ i G) G^{-1}.
  \end{align}
\endgroup

Since $W^{1,2}_\loc \cap L^\infty_\loc$ is an algebra, we see that
\begin{equation*}
  \det (a_{i j}) = a_{11}a_{22} - a_{12}a_{21} \in W^{1,2}_\loc \cap L^\infty_\loc.
\end{equation*}
Now assume in addition that the (positive) eigenvalues of $(a_{i j})$ are locally uniformly bounded away from zero, i.\,e., there exists $C>0$ such that $0<C<\lambda_1<\lambda_2$ almost everywhere in $K \cc U$. Then $\det(a_{i j})^{-1} \in L^\infty_\loc$.
Therefore, we have that $(a^{i j}) \in L^\infty_\loc$. Moreover, the fact that
\begin{equation*}
  D(\det(a_{i j})^{-1}) = - \frac{D(\det(a_{i j}))}{\det(a_{i j})^2}
\end{equation*}
implies that
\begin{equation*}
  \det(a_{i j})^{-1} \in W^{1,2}_\loc \cap L^\infty_\loc.
\end{equation*}
Hence
\begin{equation}
  (a^{i j}) \in W ^{1,2}_\loc (U, \Symp(2)) \cap L^\infty_\loc( U, \Symp(2)).
\end{equation}

Furthermore, by the boundedness of $(a^{i j})$ and as $(a_{i j}) \in W ^{1,2}_\loc$, we obtain that
\begin{equation}
  \Gamma_{i j}^k \in L^2_\loc(U).
\end{equation}
From the formula
\begin{equation*}
  A^{\frac{1}{2}} = \frac{1}{\sqrt{\tr A + 2 \sqrt{\det A}}}(A + \sqrt{\det A} I),
\end{equation*}
valid for any $A \in \Symp(2)$, we infer, using again $(a_{i j}) \in W ^{1,2}_\loc \cap L^\infty_\loc$ and the boundedness of the eigenvalues away from zero, that
\begin{equation}
  (a_{i j})^\frac{1}{2}, (a_{i j})^{-\frac{1}{2}} \in W ^{1,2}_\loc \cap L^\infty_\loc.
\end{equation}
Finally, as $\Gamma_ i \in L^2_\loc$ we conclude that
\begin{equation}
  \Omega_ i \in L^2_\loc(U, \gl(3)).
\end{equation}

It remains to show that each matrix $\Omega_ i$ is antisymmetric. (The following argument is taken from the proof of Theorem 7 in Ciarlet, Gratie, and C. Mardare \cite{ciarlet2008new}.) Equivalently, we may show that
\begin{equation}
  G \Omega_ i G = G^2 \Gamma_ i - G \partial_ i G
\end{equation}
is antisymmetric. By a direct computation,  using the symmetry of $(a_{i j})$, we find that
\begin{equation*}
  G^2 \Gamma_ i + \Gamma_ i^T G^2 = \begin{pmatrix}
    2 \Gamma_{i 11} & \Gamma_{i 12}+\Gamma_{i 21} & 0 \\
    \Gamma_{i 21}+\Gamma_{i 12} & 2 \Gamma_{i 22} & 0 \\
    0 & 0 & 0
  \end{pmatrix} = \partial_ i G^2.
\end{equation*}
Here, as usual, $\Gamma_{i j k} = a_{k \ell} \Gamma_{i j}^\ell$. We thus compute
\begin{align*}
  G \Omega_ i G &= G^2 \Gamma_ i - G \partial_ i G \\
  \nonumber &= \frac{1}{2} G^2 \Gamma_ i + \frac{1}{2} (\partial_ i G^2 - \Gamma_ i^T G^2) - G \partial_ i G \\
  \nonumber &= \frac{1}{2} (G^2 \Gamma_ i - \Gamma_ i^T G^2) + \frac{1}{2} \big((\partial_ i G)G + G \partial_ i G\big) - G \partial_ i G \\
  \nonumber &= \frac{1}{2} (G^2 \Gamma_ i - \Gamma_ i^T G^2) + \frac{1}{2} \big((\partial_ i G) G - G \partial_ i G\big),
\end{align*}
whereby, indeed, $\Omega_i \in \so(3)$.

Therefore, we have shown that if $(a_{ij}) \in W^{1,2}_\loc(U, \Symp(2)) \cap L^\infty_\loc(U, \Symp(2))$ and $(b_{ij}) \in L^2_\loc(U, \Sym(2))$ are given and the eigenvalues of $(a_{ij})$ are locally uniformly bounded from below then $\Omega \in L^2_\loc(U, \so(3) \otimes \wedge^1 \R^2)$.

\subsection{Optimal regularity theorem}
\label{sub:FTST}

We are now in a position to prove the optimal regularity case of the fundamental theorem of surface theory. By and large, we follow the proof of the corresponding Theorem 7 in Ciarlet, Gratie, and C. Mardare \cite{ciarlet2008new}.

\begin{proof}[Proof of \cref{thm:FTST}]
  We have shown in the previous section that $\Omega \in L^2(U, \so(3) \otimes \wedge^1 \R^2)$ and by assumption the compatibility equation is satisfied. Therefore, by \cref{thm:PfaffGlobal}, there exists $P \in W^{1,2}_\loc(U, \SO(3))$ such that
  \begin{equation}
    \partial_ i P = P \Omega_ i.
  \end{equation}

  Let $G_{(i)} = g_ i$ denote the $i$-th column of $G$. We know that $P \in W^{1,2}_\loc$ and $G \in W^{1,2}_\loc \cap L^\infty_\loc$. Furthermore, as $P \in \SO(3)$, $P$ is essentially bounded. Thus we conclude that $P g_ i \in W^{1,2}_\loc \cap L^\infty_\loc$.

  In order to apply \cref{lem:Poincare}, we require that
  \begin{equation*}
    \partial_ j (P g_ i) = \partial_ i (P g_ j).
  \end{equation*}
  As $\partial_ i P = P \Omega_ i$ and $P \in \SO(3)$, we obtain
  \begin{align*}
    \partial_ j (P g_ i) - \partial_ i (P g_ j) &= (\partial_ j P) g_ i + P \partial_ j g_ i - (\partial_ i P) g_ j - P \partial_ i g_ j \\
    \nonumber &= P \Omega_ j g_ i + P \partial_ j g_ i - P \Omega_ i g_ j - P \partial_ i g_ j,
  \end{align*}
  which is equal to zero if and only if
  \begin{align*}
    0 &= \Omega_ j g_ i + \partial_ j g_ i - \Omega_ i g_ j - \partial_ i g_ j \\
    \nonumber &= (G \Gamma_ j - \partial_ j G) G^{-1} g_ i + \partial_ j g_ i - (G \Gamma_ i - \partial_ i G) G^{-1} g_ j + \partial_ i g_ j \\
    \nonumber &= (G \Gamma_ j - \partial_ j G) e_ i + \partial_ j g_ i - (G \Gamma_ i - \partial_ i G) e_ j + \partial_ i g_ j \\
    \nonumber &= (G \Gamma_ j)_{(i)} - (G \Gamma_ i)_{(j)} \\
    \nonumber &= G \begin{pmatrix}
      \Gamma^1_{j i} \\ \Gamma^2_{j i} \\ b_{j i}
    \end{pmatrix}
    - G \begin{pmatrix}
      \Gamma^1_{i j} \\ \Gamma^2_{i j} \\ b_{i j}
    \end{pmatrix},
  \end{align*}
  where $e_i$ denotes the $i$-th unit vector in $\R^3$. Since $\Gamma_{i j}^k = \Gamma_{j i}^k$ and $b_{i j} = b_{j i}$, it follows that
  \begin{equation*}
    \partial_ j (P g_ i) - \partial_ i (P g_ j) = 0.
  \end{equation*}

  As a result, by \cref{lem:Poincare}, there exists $\theta \in W^{1,2}_\loc(U, \R^3)$ such that
  \begin{equation}
    \partial_ i \theta = P g_ i
  \end{equation}
  in $L^2_\loc$. Since $P g_ i \in W^{1,2}_\loc$, we conclude that in fact $\theta \in W^{2,2}_\loc(U, \R^3)$. Moreover, as the vectors $P g_ i$ are linearly independent, the map $\theta$ is an immersion.

  Defining $F := PG \in W^{1,2}_\loc \cap L^\infty_\loc$ and $f_i = F_{(i)}$ (here, $i=1,2,3$), we have that
  \begin{align*}
    \partial_ i \theta &= f_ i, \\
    F^T F = G^2 &= \begin{pmatrix}
      a_{11} & a_{12} & 0 \\ a_{21} & a_{22} & 0 \\ 0 & 0 & 1
    \end{pmatrix}.
  \end{align*}
  Thus
  \begin{equation*}
    f_ i^T f_ j = a_{i j},
  \end{equation*}
  whence
  \begin{equation}
    \partial_ i \theta \cdot \partial_ j \theta = a_{i j},
  \end{equation}
  and the matrix field $(a_{i j})$ is indeed the first fundamental form of the surface $\theta(U)$.

  Furthermore, it is clear that $f_i^T f_3 = \delta_{i3}$, $i=1,2,3$. Therefore, taking into account that $F$ is positive definite almost everywhere, it follows that
  \begin{equation*}
    f_3 = \frac{f_1 \times f_2}{\abs{f_1 \times f_2}}.
  \end{equation*}
  Meanwhile, we compute
  \begin{align*}
    \partial_{i j} \theta &= \partial_ j (P g_ i) \\
    \nonumber &= (\partial_ j P) g_ i + P \partial_ j g_ i \\
    \nonumber &= P (\Omega_ j g_ i + \partial_ j g_ i) \\
    \nonumber &= P (\Omega_ j G + \partial_ j G)_{(i)} \\
    \nonumber &= P (G \Gamma_ j)_{(i)} \\
    \nonumber &= F (\Gamma_ j)_{(i)}.
  \end{align*}
  As a result, we obtain that
  \begin{align}
    \partial_{i j} \theta \cdot f_3 &= (\partial_{i j} \theta)^T f_3 \\
    \nonumber &= \big((\Gamma_ j)_{(i)}\big)^T F^T f_3 \\
    \nonumber &= \begin{pmatrix}
      \Gamma^1_{j i} & \Gamma^2_{j i} & b_{j i}
    \end{pmatrix} \cdot e_3 \\
    \nonumber &= b_{j i},
  \end{align}
  whereby the matrix field $(b_{i j})$ is the second fundamental form of $\theta(U)$.

  Regarding the question of uniqueness of the immersion thus obtained, we note that by \cref{thm:PfaffGlobal}, the matrix field $P$ is unique up to a multiplicative constant $C \in \SO(3)$, while the function $\theta$ that results from the application of \cref{lem:Poincare} is unique up to an additive constant $b \in \R^3$. Therefore, any two immersions $\theta$, $\tilde{\theta}$ constructed by means of the above procedure are related by
  \begin{equation*}
    \theta = C \tilde{\theta} + b,
  \end{equation*}
  and the proof is complete.
\end{proof}

\subsection{Equivalence of compatibility conditions}
\label{sub:Equivalence of compatibility condition}

By means of a direct computation, we argue that \cref{eq:Compatibility} is equivalent to the Gauss--Codazzi--Mainardi equations.

\begin{proposition}\label{prop:Equivalence}
  In the $W^{2,2}_\loc$-setting of \cref{thm:FTST}, the compatibility condition in \cref{eq:Compatibility} is necessary and sufficient for the Gauss--Codazzi--Mainardi equations to hold.
\end{proposition}
\begin{proof}
  Assuming the compatibility condition, we have shown the existence of a $W^{2,2}_\loc$-immersion with associated first and second fundamental forms $(a_{ij})$, $(b_{ij})$ which necessarily satisfy the Gauss--Codazzi--Mainardi equations in the distributional sense.

  Moreover, we have
  \begin{align}
    0 =&~ \partial_1 \Omega_2 - \partial_2 \Omega_1 - \Omega_2 \Omega_1 + \Omega_1 \Omega_2 \\
    \nonumber =&~ \partial_1 \big((G \Gamma_2 - \partial_2 G) G^{-1}\big) \\
    \nonumber &- \partial_2 \big((G \Gamma_1 - \partial_1 G) G^{-1}\big) \\
    \nonumber &- (G \Gamma_2 - \partial_2 G) G^{-1} (G \Gamma_1 - \partial_1 G) G^{-1} \\
    \nonumber &+ (G \Gamma_1 - \partial_1 G) G^{-1} (G \Gamma_2 - \partial_2 G) G^{-1} \\
    \nonumber =&~ \big(\partial_1(G \Gamma_2) - \partial_1 \partial_2 G\big) G^{-1} - (G \Gamma_2 - \partial_2 G) G^{-1} (\partial_1 G) G^{-1} \\
    \nonumber &- \big(\partial_2(G \Gamma_1) - \partial_2 \partial_1 G\big) G^{-1} + (G \Gamma_1 - \partial_1 G) G^{-1} (\partial_2 G) G^{-1} \\
    \nonumber &- (G \Gamma_2 - \partial_2 G) G^{-1} (G \Gamma_1 - \partial_1 G) G^{-1} \\
    \nonumber &+ (G \Gamma_1 - \partial_1 G) G^{-1} (G \Gamma_2 - \partial_2 G) G^{-1}
  \end{align}
  if and only if
  \begin{align}
    0 =&~ \partial_1(G \Gamma_2) - \partial_1 \partial_2 G - (G \Gamma_2 - \partial_2 G) G^{-1} (\partial_1 G) \\
    \nonumber &- \partial_2(G \Gamma_1) + \partial_2 \partial_1 G + (G \Gamma_1 - \partial_1 G) G^{-1} (\partial_2 G) \\
    \nonumber &- (G \Gamma_2 - \partial_2 G) (\Gamma_1 - G^{-1} \partial_1 G) \\
    \nonumber &+ (G \Gamma_1 - \partial_1 G) (\Gamma_2 - G^{-1} \partial_2 G) \\
    \nonumber =&~ (\partial_1 G) \Gamma_2 + G \partial_1 \Gamma_2 - G \Gamma_2 G^{-1} (\partial_1 G) + (\partial_2 G) G^{-1} (\partial_1 G) \\
    \nonumber &- (\partial_2 G) \Gamma_1 - G \partial_2 \Gamma_1 + G \Gamma_1 G^{-1} (\partial_2 G) - (\partial_1 G) G^{-1} (\partial_2 G) \\
    \nonumber &- G \Gamma_2 \Gamma_1 + G \Gamma_2 G^{-1} (\partial_1 G) + (\partial_2 G) \Gamma_1 - (\partial_2 G) G^{-1} (\partial_1 G) \\
    \nonumber &+ G \Gamma_1 \Gamma_2 - G \Gamma_1 G^{-1} (\partial_2 G) - (\partial_1 G) \Gamma_2 + (\partial_1 G) G^{-1} (\partial_2 G) \\
    \nonumber =&~ G (\partial_1 \Gamma_2 - \partial_2 \Gamma_1 - \Gamma_2 \Gamma_1 + \Gamma_1 \Gamma_2).
  \end{align}
  Therefore, the compatibility condition is equivalent to
  \begin{equation}
    \partial_i \Gamma_j + \Gamma_i \Gamma_j = \partial_j \Gamma_i + \Gamma_j \Gamma_i.
  \end{equation}

  On the other hand, in Mardare \cite{mardare2005pfaff}, it has been shown that these equations are indeed equivalent to the Gauss--Codazzi--Mainardi equations, understood in the sense of distributions. We note that their argument readily carries over to the present $p=2$ case. \end{proof}

\section{A Weak Compactness Theorem for Immersions in the Class $W^{2,2}_\mathrm{loc}$}
\label{sec:Compactness theorem}

In order to prove the weak compactness theorem, we first show a corresponding statement for the Pfaffian system $\grad P = P \Omega$.

\begin{lemma}\label{lem:Compactness}
  Let $\{\Omega^k\} \subset L^2(U, \so(3) \otimes \wedge^1 \R^2)$ be a sequence such that $\Omega^k \rightharpoonup \Omega$ in $L^2$ as $k \to \infty$ and suppose that $\Omega^k$ satisfies the compatibility condition for every $k$.
  Then, up to the choice of a subsequence, there exists a sequence $\{P^k\} \subset W^{1,2}_\loc(U, \SO(3))$ of solutions to the equation $\grad P^k = P^k \Omega^k$ such that $P^k \rightharpoonup P$ in $W^{1,2}_\loc$ as $k \to \infty$ and $\grad P = P \Omega$.
\end{lemma}

\begin{proof}
  By \cref{thm:PfaffGlobal}, there exists a sequence $\{P^k\} \subset W^{1,2}_\loc(U, \SO(3))$ such that, for each $k$, $\partial_i P^k = P^k \Omega_i^k$ and $\norm{\grad P^k}_{L^2_\loc} \leq C \norm{\Omega^k}_{L^2_\loc}$.
  Then, as $P^k \in \SO(3)$ and $\{\Omega^k\}$ is uniformly bounded in $L^2_\loc$, so are $\{P^k\}$ and $\{\grad P^k\}$. As a result, there exists a subsequence, still denoted $\{P^k\}$, that converges weakly to some $P$ in $W^{1,2}_\loc$, and strongly in $L^2_\loc$.
  It remains to show that $\grad P = P \Omega$. We know that $\grad P^k \rightharpoonup \grad P$ in $L^2_\loc$. Moreover, since $P^k \to P$ and $\Omega^k \rightharpoonup \Omega$ in $L^2_\loc$ we infer that the product sequence $P^k \Omega^k$ is weakly convergent to some $v$ in $L^1_\loc$. On the other hand, since $P^k \Omega^k = \grad P^k$ for every $k$, we must have for every $\varphi \in L^\infty_\loc \subset L^2_\loc$ that
  \begin{equation*}
    \int P^k \Omega^k \varphi = \int \grad P^k \varphi \to \int \grad P \varphi = \int P \Omega \varphi,
  \end{equation*}
  whereby $v = P \Omega$, by the uniqueness of weak limits, and thus $\grad P = P \Omega$.
\end{proof}

Finally, we can prove \cref{thm:Compactness}.

\begin{proof}[Proof of \cref{thm:Compactness}]
  Let such a sequence $\{\theta^k\}$ of immersions be given. Then we denote the corresponding sequences of first and second fundamental forms by $\{(a_{ij})^k\}$, $\{(b_{ij})^k\}$, respectively. By assumption, we have that $(a_{ij})^k \in W^{1,2}_\loc(U, \Symp(2)) \cap L^\infty_\loc(U, \Symp(2))$ and $(b_{ij})^k \in L^2_\loc(U, \Sym(2))$.
  Moreover, for each $k$, we may define $\Omega_{i}^k \in L^2_\loc(U, \so(3))$ as in \cref{sub:RegularityCoefficients}.

  For each $k$, the $\Omega_i^k$ necessarily satisfy the compatibility equation, \cref{eq:Compatibility} (the proof of Theorem 1 of Ciarlet, Gratie, and C. Mardare \cite{ciarlet2008new} carries over to the present $p=2$ case).
  Furthermore, it is straightforward to see from the estimates in \cref{sub:RegularityCoefficients} that the sequence $\{\Omega_i^k\}$ is uniformly bounded in $L^2_\loc$ and thus subsequentially weakly convergent to some limit $\Omega_i \in L^2_\loc(U, \so(3))$.   By \cref{lem:Compactness}, therefore, up to the choice of a subsequence, there exists a sequence $\{P^k\} \subset W^{1,2}_\loc(U, \SO(3))$ of solutions to the equation $\grad P^k = P^k \Omega^k$ such that $P^k \rightharpoonup P$ in $W^{1,2}_\loc$ as $k \to \infty$ and $\grad P = P \Omega$.
  Since $\partial_j \partial_i P = \partial_i \partial_j P$ we thus have that $\partial_j (P \Omega_i) = \partial_i (P \Omega_j)$, which shows after a short computation that the compatibility equation is satisfied by the weak limit $\Omega_i$.

  At the same time, the uniformly bounded sequences $\{(a_{ij})^k\}$, $\{(b_{ij})^k\}$ possess subsequences that are weakly convergent to some $(a_{ij})$, $(b_{ij})$ in $W^{1,2}_\loc$ and $L^2_\loc$, respectively.     They satisfy $(a_{ij}) \in W^{1,2}_\loc(U, \Symp(2)) \cap L^\infty_\loc(U, \Symp(2))$ and $(b_{ij}) \in L^2_\loc(U, \Sym(2))$ and the eigenvalues of $(a_{ij})$ are uniformly bounded from below in $U$.
  As a result, we have that $\Omega_i$ and the components of the connection form induced by $(a_{ij})$ and $(b_{ij})$ coincide.   Hence we obtain from \cref{thm:FTST} an immersion $\theta \in W^{2,2}_\loc(U, \R^3)$ with first and second fundamental forms $(a_{ij})$ and $(b_{ij})$, respectively.
  On the other hand, the given sequence $\{\theta^k\}$ must have a weakly convergent subsequence in $W^{2,2}_\loc$ with a weak limit $\bar{\theta}$, which coincides with the immersion $\theta$ due to the uniqueness of distributional limits. \end{proof}


\begin{bibdiv}
\begin{biblist}

\bib{chen2018global}{article}{
      author={Chen, Gui-Qiang~G.},
      author={Li, Siran},
       title={Global weak rigidity of the {Gauss--Codazzi--Ricci} equations and
  isometric immersions of {Riemannian} manifolds with lower regularity},
        date={2018},
     journal={J. Geom. Anal.},
      volume={28},
      number={3},
       pages={1957\ndash 2007},
}

\bib{ciarlet2008new}{article}{
      author={Ciarlet, Philippe~G.},
      author={Gratie, Liliana},
      author={Mardare, Cristinel},
       title={A new approach to the fundamental theorem of surface theory},
        date={2008},
     journal={Arch. Ration. Mech. Anal.},
      volume={188},
      number={3},
       pages={457\ndash 473},
}

\bib{clelland2017frenet}{book}{
      author={Clelland, Jeanne~N.},
       title={From {Frenet} to {Cartan}: {The} method of moving frames},
      series={Graduate Studies in Mathematics},
   publisher={American Mathematical Society},
     address={Providence, RI},
        date={2017},
      volume={178},
}

\bib{coifman1993compensated}{article}{
      author={Coifman, R.},
      author={Lions, P.-L.},
      author={Meyer, Y.},
      author={Semmes, S.},
       title={Compensated compactness and {Hardy} spaces},
        date={1993},
     journal={J. Math. Pures Appl.},
      volume={72},
      number={9},
       pages={247\ndash 286},
}

\bib{hartman1950fundamental}{article}{
      author={Hartman, Philip},
      author={Wintner, Aurel},
       title={On the fundamental equations of differential geometry},
        date={1950},
     journal={Amer. J. Math.},
      volume={72},
      number={4},
       pages={757\ndash 774},
}

\bib{laurain2018optimal}{article}{
      author={Laurain, Paul},
      author={Rivière, Tristan},
       title={Optimal estimate for the gradient of {Green's} function on
  degenerating surfaces and applications},
        date={2018},
     journal={Comm. Anal. Geom.},
      volume={26},
      number={4},
       pages={887\ndash 913},
}

\bib{mardare2003fundamental}{article}{
      author={Mardare, Sorin},
       title={The fundamental theorem of surface theory for surfaces with
  little regularity},
        date={2003},
     journal={J. Elasticity},
      volume={73},
      number={1--3},
       pages={251\ndash 290},
}

\bib{mardare2005pfaff}{article}{
      author={Mardare, Sorin},
       title={On {Pfaff} systems with {$L^p$} coefficients and their
  applications in differential geometry},
        date={2005},
     journal={J. Math. Pures Appl.},
      volume={84},
      number={12},
       pages={1659\ndash 1692},
}

\bib{mardare2007systems}{article}{
      author={Mardare, Sorin},
       title={On systems of first order linear partial differential equations
  with {$L^p$} coefficients},
        date={2007},
     journal={Adv. Differential Equations},
      volume={12},
      number={3},
       pages={301\ndash 360},
}

\bib{mardare2008poincare}{article}{
      author={Mardare, Sorin},
       title={On {Poincaré} and de {Rham's} theorems},
        date={2008},
     journal={Rev. Roumaine Math. Pures Appl.},
      volume={53},
      number={5--6},
       pages={523\ndash 541},
}

\bib{muller2009boundary}{article}{
      author={Müller, Frank},
      author={Schikorra, Armin},
       title={Boundary regularity via {Uhlenbeck--Rivière} decomposition},
        date={2009},
     journal={Analysis (Munich)},
      volume={29},
      number={2},
       pages={199\ndash 220},
}

\bib{riviere2007conservation}{article}{
      author={Rivière, Tristan},
       title={Conservation laws for conformally invariant variational
  problems},
        date={2007},
     journal={Invent. Math.},
      volume={168},
      number={1},
       pages={1\ndash 22},
}

\bib{schikorra2010remark}{article}{
      author={Schikorra, Armin},
       title={A remark on gauge transformations and the moving frame method},
        date={2010},
     journal={Ann. Inst. H. Poincaré Anal. Non Linéaire},
      volume={27},
      number={2},
       pages={503\ndash 515},
}

\bib{wente1969existence}{article}{
      author={Wente, Henry~C.},
       title={An existence theorem for surfaces of constant mean curvature},
        date={1969},
     journal={J. Math. Anal. Appl.},
      volume={26},
      number={2},
       pages={318\ndash 344},
}

\bib{wente1975differential}{article}{
      author={Wente, Henry~C.},
       title={The differential equation {$\Delta x = 2H(x_u \wedge x_v)$} with
  vanishing boundary values},
        date={1975},
     journal={Proc. Amer. Math. Soc.},
      volume={50},
      number={1},
       pages={131\ndash 137},
}

\end{biblist}
\end{bibdiv}
\end{document}